\documentclass[12pt,reqno]{amsart} 

\usepackage{amsmath, amssymb} %% AMS symbol definitions
\usepackage{amscd}            %% for Commutative Diagrams
\usepackage{amsxtra}          %% defines \spcheck, \spdot, etc.
\usepackage{upref}            %% cross-refs to sections are upright
\usepackage{xypic}
\usepackage{amsfonts, amsthm}
\usepackage{graphicx}
\usepackage[mathscr]{eucal}

\theoremstyle{plain}

\newtheorem{theorem}{Theorem}
\newtheorem{proposition}[theorem]{Proposition}

\newtheorem{corollary}[theorem]{Corollary}
\newtheorem{lemma}[theorem]{Lemma}

\newtheorem{definition}[theorem]{Definition}

\theoremstyle{remark}
\newtheorem{remark}[theorem]{Remark}

\numberwithin{equation}{section}
\numberwithin{theorem}{section}
\newcommand{\be}%
  {\protect\setcounter{equation}{\value{subsubsection}}}  
\newcommand{\ee}%
  {\protect\setcounter{subsubsection}{\value{equation}}}

\newcommand{\JL}{{\rm JL}}
\newcommand{\GL}{{\rm GL}}
\newcommand{\Sp}{{\rm Sp}}

\newcommand{\A}{\mathbb{A}}

\newcommand{\C}{\mathbb{C}} 
\newcommand{\F}{\mathbb{F}}

\newcommand{\OO}{\mathcal{O}}

\begin{document}

\title {On Symplectic Periods for Inner forms of $\GL_n$}
\author{Mahendra~Kumar~Verma} %% This is the correct form!
\address{Department of Mathematics \\ 
         Indian Institute of Technology Bombay}
\email{mahendra@math.iitb.ac.in}

\begin{abstract}
In this paper we study the question of determining when an irreducible admissible representation of
$\GL_n(D)$ admits a symplectic model, that is when such a representation has a linear functional 
invariant under $\Sp_n(D)$, where $D$ is a quaternion division algebra over a non-Archimedian local field $k$ and 
$\Sp_{n}(D)$ is the unique non-split inner form of the symplectic group $\Sp_{2n}(k)$.
We show that if a representation has a symplectic model it is necessarily unique. 
For $\GL_2(D)$ we completely classify those representations which have a symplectic model. 
Globally, we show that if a discrete automorphic representation of $\GL_{n}(D_\mathbb{A})$ has a non-zero period for $\Sp_{n}(D_\mathbb{A})$,
then its Jacquet-Langlands lift also has a non-zero symplectic period. A somewhat striking difference between distinction question for $\GL_{2n}(k)$, and $\GL_n(D)$(with respect 
to $\Sp_{2n}(k)$ and $\Sp_n(D)$ resp.) is that there are supercuspidal representations of $\GL_n(D)$ which are distinguished by $\Sp_n(D)$. 
The paper ends by formulating a general question classifying all unitary distinguished representations of $\GL_n(D)$, and proving a part of the local conjectures through a global conjecture.
\end{abstract}
 \maketitle

\tableofcontents

\section{Introduction}
Let $G$ be a group and $H$ a subgroup of $G$. We recall that a complex representation $\pi$ of $G$  is said to be
$H$-distinguished if 
\[
\mathrm{Hom}_H\left(\pi,\mathbb{C}\right)\neq 0,
\]
where $\mathbb{C}$ denotes the trivial representation of $H$. When $G=\GL_{2n}(k)$, and  $H=\Sp_{2n}(k)$, 
such representations of $\GL_{2n}(k)$ are said to have a symplectic model. 
When $k$ is a non-Archimedian local field of characteristic $0$,
and $\pi$ is an irreducible admissible complex representation of $\GL_{2n}(k)$, this question has been extensively studied by several authors starting with the work of 
M. J. Heumos and S. Rallis in \cite{HR90}. 
A rather complete classification of $\Sp_{2n}(k)$-distinguished unitary representations of $\GL_{2n}(k)$ is  due to O. Offen and E. Sayag \cite{OSE07}.

When $F$ is a number field, the analogous global question is framed in terms of the non-vanishing of certain periods
of  automorphic forms $f$ on $G(F)\setminus G(\mathbb A)$, where $\mathbb A$
is the ring of ad\`eles of $F$, given by
\[
\int_{H(F)\setminus H(\mathbb A)} f(h)dh.
\]
This question has  been settled in \cite{Offa, Off} and, in fact, Offen and Sayag treat some aspects of the local questions via global methods.

In this paper we study the irreducible admissible  representations of $\mathrm{GL}_n(D)$ which are $\Sp_n(D)$-distinguished, 
where $\mathrm{Sp}_n(D)$ is an inner form of $\mathrm{Sp}_{2n}(k)$ constructed using the unique quaternion division algebra $D$ over $k$
(we will define this more precisely in Section 2). We proceed to state the main results of this paper.
\begin{theorem}
 \noindent Let $\pi$ be an irreducible admissible representation of $\GL_n(D)$. Then 
\[   \mathrm{dim ~ Hom}_{\Sp_n(D)}\left(\pi, \C\right)\leq 1. \]
\end{theorem}

 The following theorem gives a partial answer to the question on distinction of a supercuspidal representation of $\GL_n(D)$ by $\Sp_n(D)$.
\begin{theorem}
 Let $\pi$ be a supercuspidal representation of $\GL_n(D)$ with Langlands parameter $\sigma_{\pi}=\sigma\otimes {\rm sp} _r$ where $\sigma$ is an irreducible representation of the Weil group $W_k$, and ${\rm sp} _r$ is the $r$-dimensional irreducible representation of ${\rm SL}_2(\mathbb C)$ part of the Weil-Deligne group $W_k'$. Then if $r$ is odd, $\pi$ 
is not distinguished by $\Sp_n(D)$.
\end{theorem}
In section 6, we have  constructed explicit examples of supercuspidal 
representations of $\GL_n(D)$ 
which are distinguished by $\Sp_n(D)$ 
for any odd $n\geq 1$, and in section 7 we prove a complete classification of 
discrete series representations of $\GL_n(D)$ 
which are distinguished by $\Sp_n(D)$ assuming globalization of locally distinguished representations to globally distinguished
representations together with a natural global conjecture on distinction of automorphic representations of $\GL_n(D)$ 
by $\Sp_n(D)$. 

Here is a global theorem which is a simple consequence of Offen and Sayag's work.

\begin{theorem}
Let $D$ be a quaternion division algebra over $F$ and $D_{\mathbb{A}} = D \otimes_{F} \mathbb{A}$. 
Let $\Pi$ be an automorphic representation of $\GL_{n}(D_\mathbb{A})$ which appears in the discrete spectrum
of $\GL_{n}(D_\mathbb{A})$ and has non-vanishing period integral on   $\Sp_n(D)\setminus\Sp_n(D_\mathbb{A})$. 
Let $\JL(\Pi)$ be the Jacquet-Langlands lift of $\Pi$. Then the representation ${\rm JL}(\Pi)$ of $\GL_{2n}(\mathbb{A}_F)$ has 
non-vanishing period integral on $\Sp_{2n}(F)\setminus \Sp_{2n}(\mathbb{A}_F)$.
\end{theorem}

We now briefly describe the organization of this paper. 
In Section \ref{Notation}, we set up notation and give definitions. 
In this section we define the inner forms of a symplectic group over a local field $k$. 
In Section 3, we prove the uniqueness of the symplectic model for irreducible representations of $\GL_n(D)$.
In section 4, we are able to completely analyze the question of distinction of subquotients of principal series representations
of $\GL_2(D)$ by $\Sp_2(D)$ via Mackey theory.
In Section 5, we prove that  non-vanishing of symplectic period of an irreducible discrete spectrum automorphic representation of 
 $\GL_{n}(D_\mathbb{A})$ is preserved under the Jacquet-Langlands correspondence. In this section, we partially analyze distinction problem for  supercuspidal representations of $\GL_n(D)$. In Section 6, we construct  examples of supercuspidal representations of $\GL_n(D)$
which are distinguished by $\Sp_n(D)$. The paper ends by formulating a general question classifying all unitary distinguished representations of $\GL_n(D)$, and proving a part of the local conjectures through a global conjecture.\\

\noindent {\bf Acknowledgements}. I am sincerely grateful to Dipendra Prasad for suggesting the problem and answering my
questions patiently; in particular the example in Section 6, and the conjectures in section 7 are due to him. This paper would have not been possible without his help and ideas. I also thank him for reading the article carefully many times.
I wish to thank my advisor Ravi Raghunathan for  continuous encouragement  during the preparation of this article. I am also thankful
to the Council for Scientific and Industrial Research for financial support.

\section{Notation and Definitions}\label{Notation}
\noindent Let $k$ be a non-Archimedian local field of characteristic zero, and let $D$ be the unique quaternion division 
algebra  over $k$. We denote the reduced trace and reduced norm maps on $D$ by $T_{D/k}$
and $N_{D/k}$ respectively.  Let $\tau$ be the involution on $D$ defined by 
$x\rightarrow \overline{x}=T_{D/k}(x)-x$.\\
 \noindent For $n\in\mathbb{N}$, let
\[V_n=e_1D\oplus....\oplus e_nD\ \]
be a right $D$-vector space of dimension $n$. 
\begin{definition}
We define a Hermitian form on $V_n$ by 
\begin{enumerate}
\item $(e_i,e_{n-j+1})=\delta_{ij}$ for $i=1,2,\ldots,n$;
\item $(v,v{'})=\tau(v{'},v)$;
\item $(vx,v{'}x{'})=\tau(x)(v,v{'})x{'}$,
for $v,v{'}\in V_n, x,x{'}\in D.$
\end{enumerate}
\end{definition}
\noindent Let $\Sp_n(D)$ be the group of isometries of the Hermitian form $(\cdot,\cdot)$. The group $\Sp_n(D)$ is the unique non-split inner form of 
the group $\Sp_{2n}(k)$. Clearly $\Sp_n(D)\subset \GL_n(D).$ The group $\Sp_n(D)$ can also be defined as
\[\Sp_n(D)=\left\{A\in \GL_n(D)|~AJ~{}^t\bar A=J\right\},\] where ${}^t\bar{A} = (\bar{a}_{ji})$ for $A=(a_{ij})$ and
 $$J=\begin{pmatrix}
   & & & & & 1\\
   & & & & 1 & \\
  & & & 1&  & \\
%\vdots & \vdots & \vdots &\ldots &\vdots\\
 & & .& & &\\
& .& & & &\\
1& & & & &\\
   
\end{pmatrix}$$\\
For a right $D$-vector space $V$, let $\GL_D(V)$ be the group of all invertible $D$-linear transformations on $V$. 
Similarly, let $\Sp_D(V)$ be the group of all invertible $D$-linear transformations on $V$ which preserve the above defined Hermitian form on $V$.
Let $\nu$ denote the character of $\GL_n(D)$ which is the absolute value of the reduced norm on the group $\GL_n(D).$
For any $p$-adic group $G$, let $\delta_G$ denote the modular character of $G.$ We denote the trivial representation of any group 
by $\C.$
For any representation $\pi$, we will denote its contragredient representation by $\hat{\pi}.$

\section{Uniqueness of symplectic models}

In this section we will show that for an irreducible representation $\pi$ of $\GL_n(D)$,
$\mathrm{dim~ Hom}_{\Sp_n(D)}\left(\pi,\C\right)\leq 1$.
This result is due to M. J. Heumos and S.  Rallis \cite{HR90} when $D$ is replaced by a local field $k$. Our proof is a
straightforward adaptation of their methods. We first need a result from \cite{Raghu} which gives the realization of the 
contragredient representation of an irreducible representation of $\GL_n(D)$. 
\begin{theorem} \label{Raghuram}
Let $D$ be the quaternion division algebra over $k$, $x\rightarrow \overline{x}=T_{D/k}(x)-x$ be the canonical 
anti-automorphism of order 2 on $D$. Let $G =\GL_n(D)$, and let $\sigma :G\rightarrow G$ be the automorphism of $G$ 
given by $\sigma(g)=J\left(^t\bar{g}^{-1}\right)J$, where $\bar{g}=(\overline g_{ij})$ and
$J$ is the anti-diagonal matrix with all entries 1.
Let $\pi$ be an irreducible admissible representation of $\GL_n(D)$ and $\pi^{\sigma}$ be the
representation defined
 by $\pi^{\sigma}(g)=\pi(\sigma(g))$. Then $\pi^{\sigma}={\hat{\pi}}$, where ${\hat{\pi}}$ is the
 contragredient of $\pi$.
\end{theorem}
Let $k$ be a local field of characteristic different from 2, $\bar{k}$ the  algebraic closure of $k$ and $M$ (resp. $\bar{M}$)
 denote the set of $n\times n$ matrices with coefficients in $k$ (respectively $\bar{k}$). Let $\sigma$ denote an anti-automorphism
 on $\bar{M}$ of order 2. We will record two lemmas from \cite{HR90} below.

\begin{lemma}[Lemma 2.2.1 of \cite{HR90}]
For any $A\in \GL_n(k)$, there exists a polynomial $f\in \bar{k}[t]$ such that $f(A)^2=A$.
 \end{lemma}

\begin{proposition}[Proposition 2.2.2 of \cite{HR90}]\label{Prop222}
 For any $A\in \GL_{n}(\bar{k})$, there exists $U, V\in \GL_{n}(\bar{k})$ such that $\sigma(U)=U,\sigma(V)=V^{-1}$ and $A=UV$.
\end{proposition}
 
Set $A^J=J~{}^t\bar AJ$ for $A\in \GL_n(D)$. Then $A \rightarrow  A^J$ is an anti-involution on $\GL_n(D)$ of order $2.$  By Proposition \ref{Prop222}, over an algebraically closed field, there exist $U,V\in{GL_{2n}(\bar{k})}$, such that $V^J=V^{-1}, U^J=U$ and
$A=UV.$ Then $A^J=V^JU^J=V^{-1}U=V^{-1}AV^{-1}.$ Since $V\in \Sp_{2n}(\bar k)$ if and only if $V\in \GL_{2n}(\bar k)$ and $V^J=V^{-1}$,  
$A^J$ and $A$ lie in the same double cosets over algebraic closure.

The next result shows that $A$ and $A^J$ lie in the
same double coset of $\Sp_n(D)$ in $\GL_n(D)$. 
Let us first recall a theorem due to Kneser and Bruhat-Tits.
\begin{theorem}\label{KBT}
 Let $G$ be any semi-simple simply connected group over $p$-adic field $k.$ Then $H^1(k,G)=0$.
\end{theorem}

The theorem above will be used in conjunction with our modification of Lemma 2.3.3 \cite{HR90}
given below.

\begin{proposition}\label{Prop231}  
Let $D$ be a quaternion division algebra over a local field $k$ of characteristic zero. 
Let $A\in \GL_{n}(D)$. Then there exist $P_1,P_2\in \Sp_{n}(D)$, such that $A^J=P_1AP_2$.
\end{proposition}
\begin{proof} 
Consider the set
\begin{equation*}
V(A)=\{(P_1,P_2)\in \Sp_n(D)\times \Sp_n(D)|A^J=P_1AP_2 \}.
\end{equation*}
The assertion contained in the proposition is equivalent to saying that $V(A)$ is non-empty. Clearly $V(A)$ is an algebraic subset 
of $\Sp_{2n}(\bar k)\times\Sp_{2n}(\bar k)$. Note that $A\cap A\Sp_n(D)A^{-1}$ is the subgroup of $\GL_n(D)$ which leaves the symplectic form associated with the matrix
 $J'={}^t\bar AJA^{-1}$
 invariant.
Denote the group $\Sp_n(D)\cap A\Sp_n(D)A^{-1}$ by $\Sp(J,J')$.
 Consider the right action of $\Sp(J,J')$
  on $V(A)$ by
 $R\left(P_1,P_2\right)=\left(P_1R^{-1},A^{-1}RAP_2\right)$. 
Since $P_1R^{-1}AA^{-1}RAP_2=P_1AP_2=A^J$, $(P_1R^{-1},A^{-1}RAP_2)=R(P_1,P_2)\in V(A)$,
we have,
\[
\begin{array}{rl}
 R(P_1,P_2)&=(P_1R^{-1},A^{-1}RAP_2),\\
S(R(P_1,P_2))&=(P_1R^{-1}S^{-1}, A^{-1}SAA^{-1}RAP_2),\\
&=(P_1R^{-1}S^{-1}, A^{-1}SRAP_2)
\end{array}.
\]
for $R,S\in \Sp(J,J')$ and $\left(P_1,P_2\right)\in V(A)$, verifying that we do indeed have an action.
We check that this action is fixed point free. This is because
 if $R(P_1, P_2)=(P_1, P_2)$ for $R\in \Sp(J,J')$ and $(P_1, P_2)\in V(A)$, then $P_1R^{-1}=P_1$ which gives $R=1$
 
We next check that the action is transitive. For this let $P=(P_1, P_2)$ and $Q=(Q_1, Q_2)$ be two points 
in $V(A)$. We need to prove that there exists $R\in \Sp(J,J')$ such that $RP=Q$, that is, that $R(P_1,P_2)=(Q_1,Q_2)$, 
or equivalently that
\[
(P_1R^{-1},A^{-1}RAP_2)=(Q_1, Q_2).
\]
Let $R=Q_1^{-1}P_1\in \Sp_n(D)$ then $P_1R^{-1}=Q_1$. With this choice of $R$
 $$A^{-1}RAP_2=A^{-1}Q_1^{-1}P_1AP_2=A^{-1}Q_1^{-1}Q_1AQ_2=Q_2.$$
 \indent In the second equality we have used the definition of $V(A)$ because of which $A^J=P_1AP_2=Q_1AQ_2.$
Also $P_1AP_2=Q_1AQ_2$ gives 
\[
R=Q_1^{-1}P_1=AQ_2P_2^{-1}A^{-1}\in A\Sp_n(D)A^{-1}.
\]
Hence, $R\in\Sp(J,J')$ which shows that the action of ${\Sp}(J,J')$ on $V(A)$ is transitive.
Therefore $V(A)$ is a right principal homogeneous space 
for the group $\Sp(J,J')$.

 Klyachko proved that over an algebraically closed field, ${\Sp}(J,J')$ is an extension of a product 
 of symplectic groups by a unipotent group.  Therefore, over a general field,  $\Sp(J,J')$ is an extension of a form 
 of a product of symplectic groups by a unipotent group, that is, there exists an exact sequence of algebraic groups of the form
\begin{equation*}
 1\rightarrow U \rightarrow \Sp(J,J')\rightarrow S \rightarrow 1,
\end{equation*}
with $S$, a form of a product of symplectic groups. Therefore we get the following exact sequence of  Galois cohomology sets:
\begin{equation*}
 H^1(k,U)\rightarrow H^1(k,\Sp(J,J'))\rightarrow H^1(k,S).
\end{equation*}
It is well-known that $H^1(k, U)=0$ for any unipotent group $U$ over a field of characteristic zero \cite{Serre}.
Since by Theorem \ref{KBT}, $H^1(k,S)=0$, the exact sequence above  gives $H^1(k,\Sp(J,J'))=0$. 
Since  $V(A)$ is a principal homogeneous for $\Sp(J,J')$ and $H^1(k,\Sp(J,J'))=0$, it follows that $V(A)(k)\neq \emptyset$, proving the proposition. 
\end{proof}
We recall the following result from \cite{DP90}.
\begin{lemma}\label{Prasad}
Let $G$ be an $l$-group and $H$ be a closed subgroup of $G$ such that $G/H$ carries a $G$-invariant measure. Suppose
 $x\rightarrow\bar{x}$ is an anti-automorphism of $G$ which leaves $H$ invariant and acts trivially on those distributions
 on $G$ which are $H$
 bi-invariant. Then for any smooth irreducible representation $\pi$ of $G$,
$\dim \mathrm{Hom}_{H}(\pi,\C)\cdot \dim \mathrm{Hom}_{H}(\hat\pi,\C)\leqslant 1$.
\end{lemma}
\begin{corollary}
Let $G=\GL_n(D), H=\Sp_n(D)$, and let $i$ be the anti-automorphism on $G$ given by $A\rightarrow  {}^JA^{-1}$
Then for any smooth irreducible representation $\pi$ of $G$,
$\dim \mathrm{Hom}_{\Sp_n(D)}\left(\pi,\C\right)\cdot\dim \mathrm{Hom}_{\Sp_n(D)}\left(\hat\pi,\C\right)\leqslant 1$.
\end{corollary}
\begin{proof}
The hypotheses of Lemma \ref{Prasad} follow from Proposition \ref{Prop231} by standard methods in Gelfand-Kazhdan theory.
Hence, the corollary is an immediate consequence of Lemma \ref{Prasad}.
\end{proof}
 We are now in a position to prove the main theorem of this section.
\begin{theorem}
 Let $\pi$ be an irreducible admissible representation of $\GL_n(D)$. Then
 $\dim {\rm Hom}_{\Sp_n(D)}(\pi,\C)\leqslant 1$.
\end{theorem}
\begin{proof}
Let $(\pi_1, V)$ be 
the representation defined by $\pi_1(g)=\pi(^Jg^{-1})$.\\ Let $\lambda\in {\rm Hom}_{\Sp_n(D)}\left(\pi_1,\mathbb{C}\right).$ 
Then $\lambda(\pi_1(g)v)=\lambda (v)$ which gives $\lambda(\pi (^Jg^{-1})v)=\lambda (v)$.
Since $H$ is invariant under $g\rightarrow {}^Jg^{-1}$, 
$\lambda(\pi(g)v)=\lambda(v)$ for $g\in H$, so $\lambda\in {\rm Hom}_{\Sp_n(D)}(V,\mathbb{C})$.
The other inclusion follows similarly. Therefore, $\dim {\rm Hom}_{\Sp_n(D)}\left(\pi,\mathbb{C}\right)= \dim{\rm Hom}_{\Sp_n(D)}\left(\pi_1,\mathbb{C}\right)$. Now the result
follows from Theorem \ref{Raghuram} and above corollary.
\end{proof}

\section{Local theory}

The aim of this section is to analyze the principal series representations of
$\GL_2(D)$ which have a symplectic model. This can be easily done by the usual Mackey theory which is what we do here.

\subsection{Orbits and Mackey theory}
Let $H$ and $P$ be  two closed subgroups of a group $G$ and let $(\sigma,W)$ be a smooth representation of $P.$
 We assume that $G$ and $H$
are unimodular. Also, assume that $H{\setminus G/P}$ has only two elements, that is,
the  natural action of $H$ on $G/P$ has two orbits, which we will call $O_1$ and $O_2$.

Assume without loss of generality that the orbit $O_1$ of
$H$ through $eP$ is closed  and the orbit $O_2$ is open.
  Let $H_1$ be the stabilizer in $H$ of the element $eP$ in $G/P$, then $H_1=P\cap H.$  Choose an element $x$ in $G$ 
such that the coset $xP$ lies in $O_2.$ 
Then $H_2=\textnormal{Stab}_H(xP)=H \cap xPx^{-1}$. 
Therefore,  $O_1\simeq {H/H_1}$ and $O_2\simeq {H/H_2}$.
Using Mackey theory
we obtain an exact sequence of $H$-representations:
\begin{equation*}
0\rightarrow \mathrm{ind}_{H_{2}}^{H}\sigma_{2}\rightarrow \mathrm{Ind}_{P}^{G}\sigma|_{H}
\rightarrow \mathrm{Ind}_{H_{1}}^{H}\sigma_{1}\rightarrow 0,
\end{equation*}
where 
\[
\sigma_1(h)=\left(\delta_P/\delta_{H_1}\right)^{1/2}\sigma (h)  ~~\textnormal{for}~~ h\in{H_1},
\]
and
 \[
 \sigma_2(h)=\left(\delta_P/\delta_{H_2}\right)^{1/2}\sigma (h)~~  \textnormal{for} ~~h\in{H_2}.
 \]
 The question of the existence of an $H$-invariant linear form for $\pi$
can thus be addressed by studying $H$-invariant linear forms for representations of $H$ induced from its subgroups

Now we apply the Mackey theory discussed above to the our situation for $G=\GL_2(D),H=\Sp_2(D)$ and a parabolic subgroup
$P$ of $\GL_2(D)$.

Let $V$ be a 2-dimensional Hermitian right $D$-vector space with a basis $\{e_1, e_2\}$ of $V$ with 
$(e_1, e_1)=(e_2,e_2)=0$ and $(e_1, e_2)=1.$
Let $X$ be the set of all  1-dimensional $D$-subspaces of $V.$
The group $G=\GL_D(V)$  acts naturally on $V$, and induces a transitive action on $X,$
  realizing $X$ as homogeneous space for $G.$
 Then the stabilizer of a line $W$ in $G$ is a parabolic subgroup $P$ of $G$, with $X\simeq G/P$.
Using the above basis, $\GL_D(V)$ can be identified with $\GL_2(D).$  For $W=\langle e_1\rangle$,  $P$ is the parabolic subgroup consisting  upper 
triangular matrices in $\GL_2(D)$. As  we have a Hermitian structure on $V$, $H=\Sp_D(V)\subset \GL_D(V)$.

We want to understand the space $H{\setminus G/P}.$ This space can be seen as the orbit space 
  of $H$ on the flag variety $X$. This action has two orbits. One of them, say $O_{1}$, consists of all $1$-dimensional isotropic
  subspaces of $V$ and the other, say $O_{2}$ consists of all $1$-dimensional anisotropic subspaces of $V$. Here, the one dimensional
  subspace generated by a vector $v$ is called isotropic if  $(v, v)=0$; otherwise, it is called anisotropic.
  The fact that $\Sp_D(V)$ acts transitively on $O_1$ and $O_2$ follows from Witt's theorem \cite[~page 6,~\S 9]{MVW}, together with
  the well known theorem that the reduced norm $N_{D/k}:D^{\times} \rightarrow k^{\times}$ is surjective, and as a result if a vector
  $v\in V$ is anisotropic, we can assume that in the line $\langle v \rangle=\langle v\cdot D\rangle$ generated by $v$, there exists a vector $v'$ such that
  $(v',v')=1$.
  
  It is easily seen that the stabilizer of the line $\langle e_1\rangle$ in $\Sp_D(V)$ is 
  \begin{equation*}
 P_H=\left\{\begin{pmatrix}
  a & b\\
 0 & {\bar a}^{-1}\\
 \end{pmatrix} \mid a\in D^{\times},  b\in D, a\bar b +b\bar a=0\right\}.
 \end{equation*}
 \noindent Now we consider the line $\langle e_1+e_2\rangle$  inside $O_2$. To calculate the stabilizer of this line in $\Sp_D(V),$
 note that if an isometry of $V$ stabilizes the line generated by $e_1+e_2$, it also stabilizes its orthogonal complement which is
  the line generated by $e_1-e_2$. Hence, the stabilizer of the line $\langle e_1+e_2 \rangle$ in $\Sp_D(V)$ 
 stabilizes the orthogonal decomposition of $V$ as
 \[
 V=\langle e_1+e_2 \rangle \oplus \langle e_1-e_2\rangle, 
 \]
 and also acts on the vectors $\langle e_1+e_2\rangle$ and $\langle e_1-e_2\rangle$ by scalars. Thus the stabilizer in
 $\Sp_D(V)$ of the line
 $\langle e_1+e_2\rangle$ is $ D^1\times D^1$ sitting in a natural way in the Levi $D^{\times} \times D^{\times}$ of the parabolic $P$ in $\GL_2(D).$
Here $D^1$ is the subgroup of $D^\times$ consisting of reduced norm $1$ elements in $D^{\times}$.

Now consider the principal series representation $\pi=\sigma_1\times\sigma_2:={\rm Ind}_P^{\GL_2(D)}\sigma$ of $\GL_2(D)$, where
 $\sigma=\sigma_1\otimes\sigma_2$
is an irreducible representation of $D^{\times}\otimes D^{\times}$.
We analyze the restriction of $\pi$ to $\Sp_2(D)$.
By Mackey theory, we  get the following exact sequence of $\Sp_{2}(D)$ representations 
\begin{align} \label{4.1}
 0 \rightarrow & \mathrm{ind}_{D^{1} \times D^{1}}^{\Sp_2(D)}[\left(\sigma_1\otimes\sigma_2\right)|_{D^{1}\times D^{1}}]\rightarrow 
\pi \rightarrow
&\mathrm{Ind}_{P_H}^{\Sp_2(D)} \nu^{1/2} [\left(\sigma_1\otimes\sigma_2\right)|_{M_H}] \rightarrow 0.
\end{align}
 Here $\nu$ is the character on $P_H$ given by
 \begin{equation*}
 \nu\left[\begin{pmatrix}
  a & b\\
 0 & {\bar a}^{-1}
 \end{pmatrix} \right]= {\left\vert N_{D/k}(a)\right\vert}
 \end{equation*}
 
Suppose $\pi$ has a nonzero $\Sp_2(D)$-invariant linear form. Then one of the representations in the above exact 
sequence,
\begin{equation}\label{exact}
 \mathrm{ind}_{D^{1}\times D^{1}}^{\Sp_2(D)}[\left(\sigma_1\otimes\sigma_2\right)|_{D^{1}\times D^{1}}]
\text{ or } \mathrm{Ind}_{P_H}^{\Sp_2(D)}\nu^{1/2}[\left(\sigma_1\otimes\sigma_2\right)|_{M_H}],
\end{equation}
must have an $\Sp_2(D)$-invariant form.
First, consider the case when 
\begin{equation*}
 \textnormal{Hom}_{\Sp_2(D)}(\mathrm{Ind}_{P_H}^{\Sp_2(D)} \nu^{1/2} [\left(\sigma_1\otimes\sigma_2\right)|_{M_H}], \mathbb{C}) \neq 0.
\end{equation*}
Since $H/P_H$ is compact, by Frobenius reciprocity, this is equivalent to
\begin{equation*}  
\textnormal{Hom}_{M_H}(
\nu^{1/2}\left(\sigma_1\otimes\sigma_2\right),\nu^{3/2})\neq 0.
\end{equation*}
Since $M_H=\{\left(d,\bar{d}^{-1}\right)|d\in D^{\times}\}\simeq \Delta{\left(D^{\times}\times D^{\times}\right)}$, we have
\begin{equation*}
\textnormal{Hom}_{D^{\times}}(
\left(\sigma_1\otimes\hat{\sigma_2}\right),\nu)\neq 0,
 \end{equation*}
and hence 
\begin{equation}\label{closed}
\textnormal{Hom}_{D^{\times}}\left(
\sigma_1,\sigma_2\otimes\nu\right)\neq 0,
 \end{equation}
 or $$\sigma_1 \simeq \nu\otimes\sigma_2.$$\\
Now assume  that 
\begin{equation*}
 \textnormal{Hom}_{\Sp_2(D)}(\mathrm{ind}_{D^{1}\times D^{1}}^{\Sp_2(D)}
[\left(\sigma_1\otimes\sigma_2\right)|_{D^{1}\times D^{1}}],\mathbb{C})\neq 0.
\end{equation*}
 Then by Frobenius reciprocity, this is equivalent to
\begin{equation}\label{open orbit}
\textnormal{Hom}_{D^{1}\times D^{1}}(
\left(\sigma_1\otimes\sigma_2),\mathbb{C}\right)\neq 0.
 \end{equation} 
 
\begin{lemma}
 Let $(\sigma,V)$ be a finite dimensional irreducible representation of $D^{\times}$ with $\mathrm{Hom}_{D^1}\left(V, \mathbb C\right)\neq 0.$ 
Then $\sigma$ is one dimensional.
\end{lemma}
\begin{proof}
  By a theorem due to Matsushima \cite{Naka}, $D^{1}$ is the commutator subgroup of $D^{\times}.$ Since $D^1$ is a normal subgroup of $D^{\times}$,
 $V^{D^{1}}\neq \{0\}$ is invariant under $D^{\times}$ and so by the irreducibility of $V$, $V=V^{D^{1}}.$ 
 Since $(\sigma, V)$ is an irreducible representation of $D^{\times}$, on which $D^1$ operates trivially, 
 $(\sigma, V)$ as a representation of $D^{\times}/D^{1}$
is also irreducible. Since $D^{\times}/D^{1}$ is abelian, $\sigma$ must be one dimensional.
\end{proof}
From the analysis above, we deduce that if the representation
$$\pi=\sigma_1 \times \sigma_2:= \mathrm{Ind}_P^{\GL_2(D)}\left(\sigma_1 \otimes \sigma_2 \right)$$
has an $\Sp_2(D)$-invariant linear form, then either
\begin{enumerate}
 \item $\sigma_1\simeq \sigma_2 \otimes \nu$, or
 \item both $\sigma_1$ and $\sigma_2$ are $1$-dimensional representations of $D^{\times}$, hence are of the form 
 $\sigma_1 = \chi_1\circ N_{D/k}$, $\sigma_2= \chi_2 \circ N_{D/k}$ for characters $\chi_i :k^{\times} \rightarrow \C^{\times}.$
\end{enumerate}
Further, we note that the closed orbit for the action of $\Sp_2(D)$ on $P\setminus \GL_2(D)$ contributes to a $\Sp_2(D)$-invariant
 form in the first case above, whereas it is the open orbit which contributes to a $\Sp_2(D)$-invariant linear form in the second case.
 Since the part of the representation supported on the closed orbit arises as a quotient of $\pi$, we find that in the first case 
 $\pi$ must have a $\Sp_2(D)$-invariant linear form.
 
If dim$\left(\sigma_1 \otimes\sigma_2\right)>1$, then the open orbit cannot contribute to an $\Sp_2(D)$-invariant linear form,
  and therefore we conclude that if dim$\left(\sigma_1 \otimes\sigma_2\right)>1$, then $\pi=\sigma_1 \times \sigma_2$
  has an $\Sp_2(D)$-invariant form if and only if $\sigma_1= \sigma_2 \otimes \nu$. 
  Observe that if $\pi$ has an $\Sp_2(D)$-invariant linear form, and is irreducible, then by an analogue of a 
  theorem of Gelfand-Kazhdan \cite{Ge-Ka1}
    due to Raghuram \cite{Raghu}, $\hat\pi$ too has an $\Sp_2(D)$-invariant linear form. However, if
   $\pi=\sigma_1 \times \sigma_2$, and $\pi$ is irreducible, then $\hat\pi = \hat\sigma_1 \times \hat\sigma_2,$ and if 
   $\sigma_1\simeq \sigma_2 \otimes \nu$, we get $\hat \sigma_1\simeq \hat \sigma_2 \otimes \nu^{-1}$. This means by our analysis above that the 
   representation $\hat \sigma_1 \times \hat \sigma_2$ of $\GL_2(D)$ does not carry an $\Sp_2(D)$-invariant linear form.
    Therefore, we conclude that if $\sigma_1\simeq \sigma_2 \otimes \nu$, then  $\pi=\sigma_1 \times \sigma_2$ must be reducible, which is one part 
   of the following theorem of Tadic \cite{Tadic}.
   \begin{theorem}\label{Tadic} (Tadic)
 Let $\sigma_1$ and $\sigma_2 $ be two irreducible  representations of $D^{\times}.$ 
Let $\pi=\mathrm{Ind}_P^{\GL_2(D)}\left(\sigma_{1}\otimes\sigma_{2}\right)$ be the corresponding principal series
 representation of $\GL_2(D).$  Assume  dim$\left(\sigma_1 \otimes\sigma_2\right)>1$. Then $\pi$ is reducible if and only if 
$\sigma_1\simeq \sigma_2\otimes {\nu}^{\pm 1}.$ 
If $\pi$ is reducible then it has length two. Assuming $\sigma_1=\sigma_2\otimes \nu$, we have the following non-split 
exact sequence:
$$0\rightarrow  {\rm St}(\pi)\rightarrow \pi\rightarrow {\rm Sp}(\pi)\rightarrow 0,$$
where ${\rm St}(\pi)$ is a discrete series representation called a generalized Steinberg representation of $\GL_2(D)$
and ${\rm Sp}(\pi)$ is called 
a Speh representation of $\GL_2(D)$.\\
If ${\rm dim}(\sigma_1\otimes\sigma_2)=1$, then $\pi= \chi_1\times \chi_2 $ is reducible if and only if $\sigma_1\simeq {\sigma_2}\otimes {\nu}^{\pm2}.$
If $\sigma_1=\sigma_2\otimes\nu^2$, $\pi$ has a one dimensional quotient, and the submodule is a twist of the Steinberg representation
of $\GL_2(D)$.
\end{theorem}
 In the exact sequence of $\GL_2(D)$-modules
 $$0\rightarrow  {\rm Sp}(\sigma_1)\rightarrow {\rm Ind}_P^{\GL_2(D)}(\sigma_1 \nu^{-1/2}
 \otimes \sigma_1 \nu^{1/2})\rightarrow  {\rm St}(\sigma_1)\rightarrow 0,$$
and assuming that ${\rm dim}(\sigma_1)>1$, we know by our previous analysis that 
${\rm Ind}_P^{\GL_2(D)}(\sigma_1 \nu^{-1/2} \otimes \sigma_1 \nu^{1/2})$
 does not have an $\Sp_2(D)$-invariant linear form. Therefore, from the exact sequence above, it is clear that ${\rm St}(\sigma_1)$ also 
 does not have an $\Sp_2(D)$-invariant linear form.\\
 \indent On the other hand, we know that ${\rm Ind}_P^{\GL_2(D)}(\sigma_1 \nu^{1/2} \otimes \sigma_1 \nu^{-1/2})$ does have an $\Sp_2(D)$-invariant
 linear form, and  ${\rm Ind}_P^{\GL_2(D)}(\sigma_1 \nu^{1/2} \otimes \sigma_1 \nu^{-1/2})$  fits in the following exact
 sequence:
  $$0\rightarrow  {\rm St}(\sigma_1)\rightarrow {\rm Ind}_P^{\GL_2(D)}(\sigma_1 \nu^{1/2} \otimes \sigma_1
  \nu^{-1/2})\rightarrow  {\rm Sp}(\sigma_1)\rightarrow 0.$$
  Since we have already concluded that ${\rm St}(\sigma_1)$ does not have an $\Sp_2(D)$-invariant linear form and since 
  ${\rm Ind}_P^{\GL_2(D)}(\sigma_1 \nu^{1/2} \otimes \sigma_1
  \nu^{-1/2})$ has a $\Sp_2(D)$-invariant linear form, we conclude that ${\rm Sp}(\sigma_1)$ must have an $\Sp_2(D)$-invariant linear form.\\
 \indent Having completed the analysis of $\Sp_2(D)$-invariant linear forms on representations $\pi=\sigma_1 \times \sigma_2$ with 
 dim$\left(\sigma_1 \otimes\sigma_2\right)>1$, we turn our attention to the case when $\sigma_1$ and $\sigma_2$ are both one dimensional
 representations of $D^{\times}$.
 In this case, the part of $\pi$ supported on the open orbit, which is a submodule of $\pi$, contributes to an $\Sp_2(D)$-invariant linear
 form. Suppose that $\sigma_1 \neq \sigma_2\otimes\nu$, as otherwise there is an $\Sp_2(D)$-invariant linear form arising from the closed
 orbit.
 
 Since the part of $\pi$ supported on the open orbit, that is, 
 ${\rm ind}_{D^1 \times D^1}^{\Sp_2(D)}\left(\sigma_1 \otimes \sigma_2 \right)$, 
 is a submodule of $\pi$, it is not obvious that an $\Sp_2(D)$-invariant linear form on 
 ${\rm ind}_{D^1 \times D^1}^{\Sp_2(D)}\left(\sigma_1 \otimes \sigma_2 \right)$
 will extend to an $\Sp_2(D)$-invariant linear form on $\pi$. For this, as in \cite{DP90}, we need to ensure that
 $$ {\rm Ext}_{\Sp_2(D)}^{1}[{\rm Ind}_{P_H}^{\Sp_2(D)}\nu^{1/2}\left(\sigma_1 \otimes \sigma_2 \right)\mid_{M_H},\C]=0.$$
For proving this, we recall the notion of the Euler-Poincar\'e pairing between two finite length representations of any reductive group $G$,
 defined by 
 \[
 {\rm EP}_G[\pi_1, \pi_2]= \sum_{i=0}^{r(G)} (-1)^{i}{\rm dim}~{Ext}_G^{i}[\pi_1, \pi_2],
 \]
 where $r(G)$ is the split rank of $G$ which for $\Sp_2(D)$ is $1$. Therefore, for $\Sp_2(D),$
 \[
 {\rm EP}_{\Sp_2(D)}[\pi_1, \pi_2]= {\rm dim~Hom}_{\Sp_2(D)}[\pi_1,\pi_2]-{\rm dim}~{Ext}_{\Sp_2(D)}^{1}[\pi_1, \pi_2].
 \]
 By a well known theorem, ${\rm EP}_G[\pi_1, \pi_2]=0$ if $\pi_1$ is a (not necessarily irreducible) principal series representation
 of $G.$ Therefore, we find that
 \[
 {\rm EP}_{\Sp_2(D)}[{\rm Ind}_{P_H}^{\Sp_2(D)}\nu^{1/2}\left(\sigma_1 \otimes \sigma_2 \right),\C)]=0,
 \]
and so
 \[
 \dim {\rm Hom}_{\Sp_2(D)}[{\rm Ind}_{P_H}^{\Sp_2(D)}\nu^{1/2}\left(\sigma_1 \otimes \sigma_2 \right),\C)]=
 \dim {\rm Ext}_{\Sp_2(D)}^{1}[{\rm Ind}_{P_H}^{\Sp_2(D)}\nu^{1/2}(\sigma_1 \otimes \sigma_2), \C].
 \]
 Since we are assuming that $\sigma_1\neq\sigma_2\otimes\nu$,
 $$\dim {\rm Hom}_{\Sp_2(D)}[{\rm Ind}_{P_H}^{\Sp_2(D)}\nu^{1/2}\left(\sigma_1 \otimes \sigma_2 \right),\C)]=0.$$
 Therefore we conclude that 
 $${\rm Ext}_{\Sp_2(D)}^{1}[{\rm Ind}_{P_H}^{\Sp_2(D)}\nu^{1/2}(\sigma_1 \otimes \sigma_2), \C]=0.$$
 As a result, we now have proved that if $\sigma_1$ and $\sigma_2$ are one dimensional representations of $D^{\times}$, with 
 $\sigma_1\neq\sigma_2\otimes\nu$, then $\pi=\sigma_1\times\sigma_2$ has a $\Sp_2(D)$-invariant linear
  form. 

We  have proved most of  the following theorem, which we will now complete.

\begin{theorem}\label{indrepthm}
 The only subquotients of a principal series representation 
   $\pi=\sigma_1 \times \sigma_2:= \mathrm{Ind}_P^{\GL_2(D)}\left(\sigma_1 \otimes \sigma_2 \right)$ of $\GL_2(D)$ which have a $\Sp_2(D)$-
   invariant linear form are the following.
   \begin{enumerate}
    \item  When ${\rm dim}\left(\sigma_1 \otimes \sigma_2 \right)>1$, the unique irreducible quotient of 
the principal series representation    ${\rm Ind}_P^{\GL_2(D)}(\sigma \nu^{1/2} \otimes \sigma \nu^{-1/2})$ denoted by ${\rm Sp}(\sigma)$. 
    \item  When ${\rm dim}(\sigma_1)={\rm dim}(\sigma_2)=1,$ any of the irreducible principal series representations
    ${\rm Ind}_P^{\GL_2(D)}(\sigma_1  \otimes \sigma_2 ),$  whenever $\sigma_1\neq\sigma_2\otimes\nu^{\pm 2}$.
    \item When ${\rm dim}(\sigma_1)={\rm dim}(\sigma_2)=1,$ and 
    $\sigma_1=\sigma_2\otimes\nu^{2}$, the principal series representation
    ${\rm Ind}_P^{\GL_2(D)}(\sigma_1 \nu \otimes \sigma_1 \nu^{-1})$ fits in the following exact sequence:
    $$0\rightarrow  {\rm St}\otimes \chi \rightarrow {\rm Ind}_P^{\GL_2(D)}(\chi \nu\otimes \chi\nu^{-1}) \rightarrow \C_{\chi}\rightarrow 0,$$
   where $\C_{\chi}$ is the one dimensional representation of $\GL_2(D)$ on which $\GL_2(D)$ operates by the character $\chi\circ N_{D/k}$,
   $N_{D/k}$ is the reduced norm map and  ${\rm St}$ is the Steinberg representation of $\GL_2(D).$ 
 The only subquotient of ${\rm Ind}_P^{\GL_2(D)}(\chi \nu\otimes \chi\nu^{-1})$ having $\Sp_2(D)$-invariant linear form is $\C_{\chi}.$
 \end{enumerate}
  \end{theorem}

    \begin{proof}
    The only part of this theorem  not shown by the arguments above is that 
    $${\rm Hom}_{\Sp_2(D)}[{\rm St},\C]=0,$$
    where St is the Steinberg representation of $\GL_2(D)$,  an irreducible admissible representation of $\GL_2(D)$ fitting
 in the exact sequence
    $$0\rightarrow  {\rm St}\rightarrow {\rm Ind}_P^{\GL_2(D)}(\nu\otimes\nu^{-1})\rightarrow \C\rightarrow 0.$$
    Applying ${\rm Hom}_{\Sp_2(D)}[-,\C]$ to this exact sequence, we have:
\begin{align*}
  & 0\rightarrow  {\rm Hom}_{\Sp_2(D)}[\C,\C]\rightarrow {\rm Hom}_{\Sp_2(D)}[{\rm Ind}_P^{\GL_2(D)}(\nu\otimes\nu^{-1}),\C]
      \rightarrow {\rm Hom}_{\Sp_2(D)}[{\rm St},\C]\\
& \rightarrow {\rm Ext}_{\Sp_2(D)}^1[\C,\C] \rightarrow \cdots.
\end{align*}
     However, it is easy to see that  ${\rm Ext}_{\Sp_2(D)}^1[\C,\C]=0.$ Therefore, we have a short exact sequence
     $$0\rightarrow  \C\rightarrow {\rm Hom}_{\Sp_2(D)}[{\rm Ind}_P^{\GL_2(D)}(\nu\otimes\nu^{-1}),\C]
      \rightarrow {\rm Hom}_{\Sp_2(D)}[{\rm St},\C] \rightarrow 0.$$
     Hence, if ${\rm Hom}_{\Sp_2(D)}[{\rm St},\C] \neq 0$, ${\rm dim~Hom}_{\Sp_2(D)}[{\rm Ind}_P^{\GL_2(D)}(\nu\otimes\nu^{-1}),\C]\geqslant 2.$
     However, by the analysis with Mackey theory done above, we know that 
     ${\rm dim~Hom}_{\Sp_2(D)}[{\rm Ind}_P^{\GL_2(D)}(\nu\otimes\nu^{-1}),\C]=1$. Thus we have proved that
     \[{\rm Hom}_{\Sp_2(D)}[{\rm St},\C]=0. \qedhere\]
 \end{proof}
 \begin{remark}\label{remark1}
 As an important corollary of the  theorem above, note that the irreducible principal series representation
$\pi =\chi_1\times\chi_2:= {\rm Ind}_P^{\GL_2(D)}(\chi_1\otimes\chi_2)$ for characters $\chi_1$ and $\chi_2$ of $D^{\times}$
 which arise from the characters $\chi_1$ and $\chi_2$ of $k^{\times}$ via the reduced norm map of $D^{\times}$ to $k^{\times},$
 with $\chi_1 \chi_2^{-1} \neq \nu^{\pm 2}$, the representation $\pi$ is distinguished by $\Sp_2(D).$ However ${\rm JL}(\pi),$
  a representation of $\GL_4(k)$ is the irreducible principal series representation 
  ${\rm JL}(\pi)={\rm Ind}_P^{\GL_4(k)}(\chi_1 {\rm St}_2\otimes \chi_2 St_2)$ where ${\rm St}_2$ denote the Steinberg representation of
  $\GL_4(k).$ Since ${\rm JL}(\pi)$ is a generic representation of $\GL_4(k)$, it is not distinguished by $\Sp_4(k)$.
  Thus Jacquet-Langlands correspondence for representations of $\GL_2(D)$ to $\GL_4(k)$ does not always preserve distinction.
  \end{remark}

  \section {Global theory}
Let $F$ be a number field  and $ D$ be a quaternion  division algebra over $F.$
For each place $v$ of $F$, let $F_v$ be the completion of $F$ at $v.$ We can define $\GL_n(D)$ and $\Sp_n(D)$ as in the local case in the Section \ref{Notation}.

 Let $\mathbb A$ be the ring of ad\`eles of $F.$ Let  $D_v=D\otimes_FF_v$ and $D_\mathbb{A}=D\otimes_F\mathbb{A}.$  Then we can consider 
 topological groups $\GL_n(D_v)$, $\Sp_n(D_v)$, $\GL_{n}(D_\mathbb{A}),\Sp_{n}(D_\mathbb{A}),\\
\GL_{n}(\mathbb{A}_F), 
\Sp_{n}(\mathbb{A}_F).$
For an automorphic representation $\Pi$ of $\GL_{n}(D_\mathbb{A})$, 
we denote by $\JL(\Pi)$, its Jacquet-Langlands lift to $\GL_{2n}(\mathbb{A}_F).$ 

In this section, we will prove that the a non-vanishing symplectic period of a discrete automorphic representation is taken to a
non-vanishing period by the Jacquet-Langlands correspondence. 
In \cite{Off}, Offen studied the symplectic periods on the discrete  automorphic representations of 
$\GL_{2n}(\mathbb{A}_F)$.
For an automorphic form $f$ in the discrete spectrum  of $\GL_{2n}(\mathbb{A}_F)$, consider the period integral
$$\int_{\Sp_{2n}(F)\setminus \Sp_{2n}(\mathbb{A}_F)} f(h)dh.$$ 

We say that an irreducible, discrete automorphic representation $\Pi$ of $\GL_{2n}(\mathbb{A}_F)$ is 
$\Sp_{2n}(\mathbb{A}_F)$-distinguished if the above period integral is not identically zero on the space of $\Pi.$
We now recall a result from \cite{OSE08} that we will use in this section.
\begin{theorem}\label{local-global}
 Let $F$ be a number field and let $\Pi=\otimes'_v \Pi_v$ be an irreducible automorphic representation 
 of $\GL_{2n}(\mathbb{A}_F)$ in the discrete spectrum. Then the following are equivalent:
\begin{enumerate}
 \item $\Pi$ is $\Sp_{2n}(\mathbb{A}_F)$-distinguished,
\item $\Pi_v$ is $\Sp_{2n}(F_v)$-distinguished for all places $v$ of $F$,
\item $\Pi_{v_0}$ is $\Sp_{2n}(F_v)$-distinguished for some finite place $v_0$ of $F$,
\end{enumerate}
\end{theorem}

Jacquet and Rallis have shown in \cite{JR92b}, that the symplectic period  vanishes for a cuspidal automorphic representation 
of $\GL_{2n}(\mathbb{A}_F)$, that is 
$$\int_{\Sp_{2n}(F)\setminus \Sp_{2n}(\mathbb{A}_F)} f(h)dh=0.$$ 
In the next theorem, in the spirit of Jacquet-Rallis result mentioned above, we prove that those cuspidal automorphic 
representations $\Pi$ of  $\GL_{n}(D_\mathbb{A})$ for which $\JL(\Pi)$ is a cuspidal automorphic representation of $\GL_{2n}(\mathbb{A}_F)$,
have vanishing symplectic periods.
\begin{theorem}
 Suppose that  $\Pi$ is a cuspidal automorphic representation of $\GL_{n}(D_\mathbb{A})$ whose Jacquet-Langlands 
lift $\JL(\Pi)$ to $\GL_{2n}(\mathbb{A}_F)$ is cuspidal then the symplectic period integrals of $\Pi$ vanish identically.
\end{theorem}
\begin{proof} 
Assume if possible that $\Pi$ has a non-zero symplectic period.  Then $\Pi_v$ has a non-zero symplectic period for all places 
$v$ of $F.$ The representations $\JL(\Pi)$ and $\Pi$ are the same at all places $v$ of $F$ where $ D$ splits and therefore by 
the Theorem 3.2.2 of \cite{HR90}, $\Pi_v$ is not generic for any $v$ where $D$ splits. Since a cuspidal automorphic 
representation of $\GL_{2n}(\mathbb{A}_F)$ is globally generic, the local representations $\Pi_v$ are locally generic for all $v$, 
which gives a contradiction.
 \end{proof}

\begin{theorem}\label{global}
    If $\Pi$ is an automorphic representation of $\GL_{n}(D_\mathbb{A})$ which appears in the discrete spectrum, and is distinguished
 by  $\Sp_{n}(D_\mathbb{A})$ then $\JL(\Pi)$, which is an automorphic representation of $\GL_{2n}(\mathbb{A}_F)$, is globally distinguished by 
 $\Sp_{2n}(\mathbb{A}_F).$
\end{theorem}
\begin{proof}
  If $\Pi$ is $\Sp_{n}(D_\mathbb{A})$-distinguished, then it is locally distinguished at all places $v$ of $F$. Also we know that 
  $ D$ splits at almost all places of $F$ so $\Pi_v=\JL(\Pi)_v$ at almost all places of $F$.
 By  Theorem \ref{local-global}, global distinction of Jacquet-Langland lift $\JL(\Pi)$ is a 
consequence of local distinction at any place $v$ of $F$ which we know. 
\end{proof}

\begin{remark}
 If $\Pi$ is a global automorphic representations of $\GL_2(D_\mathbb{A})$ which is distinguished by $\Sp_2(D_\mathbb{A})$
 with a local component $\Pi_v=\chi_1 \times \chi_2$, a representation of $\GL_2(D_v)$ for characters $\chi_1,\chi_2:D_v^{\times} \rightarrow \C^{\times}$,
 then $\JL(\Pi)$, an automorphic representation of $\GL_4({\mathbb A}_F)$, must be distinguished by $\Sp_4({\mathbb A}_F)$
 by Theorem  \ref{global}. Since $\JL(\Pi_v)= \chi_1\circ{\rm St} \times \chi_2\circ {\rm St}$ as a representation of 
 $\GL_4(k_v),$ this seems to be in contradiction to the fact that $\JL(\Pi)$ is globally distinguished by $\Sp_4({\mathbb A}_F)$. The source
 of this apparent contradiction is the fact that in this case, $\JL(\Pi)_v=\chi_1 \times \chi_2$ as a representation
 of $\GL_4(k_v)$, as follows from the work of Badulescu.
\end{remark}
A supercuspidal representation of $\GL_{2n}(k)$ is not distinguished by $\Sp_{2n}(k)$. The situation in the case of $\GL_n(D)$
is different, that is, it may happen that a supercuspidal representation of $\GL_n(D)$ is distinguished by $\Sp_n(D)$.  We have an example of  distinguished supercuspidal representations due to Dipendra Prasad in  the next section. The following theorem gives a partial answer to the question on distinction of a supercuspidal representation of $\GL_n(D)$ by $\Sp_n(D)$.
\begin{theorem}
 Let $\pi_v$ be a supercuspidal representation of $\GL_n(D_v)$ with Langlands parameter $\sigma_{\pi_v}=\sigma\otimes {\rm sp}_r$ where $\sigma$ is an irreducible representation of the Weil-group $W_k$, and ${\rm sp}_r$ is the $r$-dimensional irreducible representation of ${\rm SL}_2(\mathbb C)$ part of the Weil-Deligne group $W_k'$. Then if $r$ is odd, $\pi_v$ is not distinguished by $\Sp_n(D_v)$.
\end{theorem}
\begin{proof}
 Assuming $r$ is odd, we prove that $\pi_v$ is not distinguished by $\Sp_n(D_v)$. Using a theorem of \cite{PS08}, we globalize $\pi_v$ to be globally distinguished automorphic representation $\Pi$ of $\GL_n(D_\mathbb{A})$ where $D$ is a global division algebra over a number field $F$
such that $F_v=k$, and $D\otimes F_v=D_v$.

Using the Jacquet-Langlands correspondence of Badulescu, we get an automorphic representation $\JL(\Pi)$ of $\GL_{2n}(\mathbb{A}_F)$ which is locally distinguished by  $\Sp_{2n}(F_w)$ at all places $w$ of $F$ where $D$ splits. By a theorem of Offen-Sayag, $\JL(\Pi)$ is globally distinguished by  $\Sp_{2n}(\mathbb{A}_F)$. By work of Badulescu, $\JL(\Pi)_v$ is one of the following
\begin{enumerate}
 \item $\JL(\Pi)_v=\JL(\Pi_v)$, a discrete series representation, or
\item $\JL(\Pi)_v=$ a Speh representation with Langlands parameter  
\[\sigma\otimes(\nu^{(r-1)/2}\oplus\nu^{(r-3)/2}\oplus\dots\oplus \nu^{-(r-1)/2}).
\]

\end{enumerate}
The first choice being a discrete  series representation, in particular generic, is never distinguished by $\Sp_{2n}(F_v)$. The fact that the second choice is also not distinguished by $\Sp_{2n}(F_v)$ uses that $r$ is odd, and is consequence of a theorem of Offen-Sayag about them.
\end{proof}

\begin{remark} The only place we used  supercuspidality of the  representation $\pi_v$ of $\GL_n(D_v)$ with Langlands parameter $\sigma_{\pi_v}=\sigma\otimes {\rm sp}_r$ where $\sigma$ is an irreducible representation of the Weil-group $W_k$, and ${\rm sp}_r$ is the $r$-dimensional irreducible representation of the ${\rm SL}_2(\mathbb C)$ part of the Weil-Deligne group $W_k'$ is in the globalization theorem of \cite{PS08}. If we grant ourselves such a globalization theorem for discrete series too, then we have the same
conclusion as in the theorem. 
\end{remark}

The theorem below together with local analysis done in Section 4 completes the distinction problem for $\GL_2(D)$.
\begin{theorem}
 No discrete series representation of $\GL_2(D_v)$ is distinguished by $\Sp_2(D_v)$.
\end{theorem}
\begin{proof}
 By our local analysis, we know this already for those discrete series representations of $\GL_2(D_v)$ which are not supercuspidal. By the previous theorem, we also know that no supercuspidal representation of $\GL_2(D_v)$ is distinguished by $\Sp_2(D_v)$ as long as its Langlands parameter is not of the form $\sigma_{\pi}=\sigma\otimes {\rm sp}_r$ where $r=2,4$. But by the work of Badulescu (cf. Proposition 7.2 below), such Langlands parameter correspond to non-supercuspidal discrete series representations of $\GL_2(D_v)$, completing the proof of theorem.

\end{proof}

\section{Explicit examples of supercuspidals with symplectic period}

In this section we construct examples of supercuspidal representations of $\GL_n(D)$ which are distinguished by $\Sp_n(D)$ for any odd $n\geq 1$.

Recall that $\OO_D$ is the maximal compact subring of $D$ with $\pi_D$ a uniformizing parameter of $\OO_D$, and $\OO_D/\langle\pi_D \OO_D\rangle \simeq \mathbb{F}_{q^2}$ where $\mathbb{F}_q$ is the residue field of $k$. The anti-automorphism $x\rightarrow \bar{x}$ of  $D$ preserve $\OO_D$
and acts as the Galois involution of $\mathbb{F}_{q^2}$ over $\mathbb{F}_q$.

Recall also that we have defined $\Sp_n(D)$ to be the subgroup of $\GL_n(D)$ by:
\[\Sp_n(D)=\left\{A\in \GL_n(D)|~AJ~{}^t\bar A=J\right\},\] where ${}^t\bar{A} = (\bar{a}_{ji})$ for $A=(a_{ij})$ and
 $$J=\begin{pmatrix}
   & & & & & 1\\
   & & & & 1 & \\
  & & & 1&  & \\
%\vdots & \vdots & \vdots &\ldots &\vdots\\
 & & .& & &\\
& .& & & &\\
1& & & & &\\
   
\end{pmatrix}.$$\\
It follows that $\Sp_n(\OO_D)\subset \GL_n(\OO_D)$, and taking the reduction of these compact groups  modulo $\pi_D$, we have:
\[ {\rm U}_n(\mathbb{F}_q)\hookrightarrow \GL_n(\mathbb{F}_{q^2}),
\] 
where ${\rm U}_n$ is defined using the Hermitian form 
$$J=\begin{pmatrix}
   & & & & & 1\\
   & & & & 1 & \\
  & & & 1&  & \\
%\vdots & \vdots & \vdots &\ldots &\vdots\\
 & & .& & &\\
& .& & & &\\
1& & & & &\\
   
\end{pmatrix}.$$

\begin{proposition}
 Let $\pi_{00}$ be an irreducible cuspidal representation of $\GL_n(\mathbb{F}_{q})$, $n$ an odd integer, and $\pi_{0}={\rm BC}(\pi_{00})$ be the base change of $\pi_{00}$ to $\GL_n(\mathbb{F}_{q^2})$. Using the reduction mod $\pi_D:\GL_n(\OO_D)\rightarrow \GL_n(\mathbb{F}_{q^2})$,
we can lift $\pi_0$ to an irreducible 
representation of $\GL_n(\OO_D)$ to be denoted by $\pi_{0}$ again. Let $\chi$ be a character of $k^{\times}$ which matches with the central character of $\pi_{0}$ on $\OO_k^{\times}$. Then 
\[ \pi= {\rm ind}_{k^{\times}\GL_n(\OO_D)}^{\GL_n(D)}(\chi \cdot\pi_{0})
\]
is an irreducible supercuspidal representation of $\GL_n(D)$ which is distinguished by $\Sp_n(D)$.
\end{proposition}
\begin{proof}
 The fact that $\pi$ is an irreducible supercuspidal representation of $\GL_n(D)$ is a well-known fact about compact induction valid in a great generality once we have checked that $\pi_{0}={\rm BC}(\pi_{00})$ is a cuspidal representation. This assertion on $\GL_n(\mathbb{F}_{q^2})$ follows from the fact that $n$ is odd in which case  we have a diagram of fields:

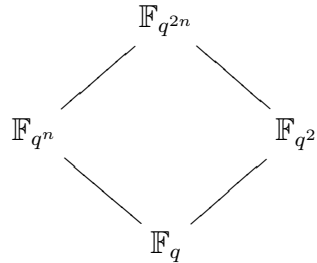
\begin{displaymath}
\xymatrix{
& \F_{q^{2n}} \ar@{-}[rd] \ar@{-}[ld] & \\
\F_{q^{n}} \ar@{-}[rd] & & \F_{q^2} \ar@{-}[ld] \\
& \F_{q} &
}
\end{displaymath}

In particular, 
\[ {\rm Gal}(\mathbb{F}_{q^{2n}}/\mathbb{F}_q)={\rm Gal}(\mathbb{F}_{q^{n}}/\mathbb{F}_q)\times {\rm Gal}(\mathbb{F}_{q^{2}}/\mathbb{F}_{q}).
\]
Thus given a character $\chi_{00}:\mathbb{F}_{q^n}^{\times}\rightarrow \mathbb{C}^{\times}$ whose Galois conjugate are distinct (and which gives rise to the cuspidal 
representation $\pi_{00}$ of $ \GL_n(\mathbb{F}_{q})$),
the character $\chi_{0}:\mathbb{F}_{q^{2n}}^{\times} \rightarrow \mathbb{C}^{\times}$ obtained from $\chi_{00}$ using the norm map:
$\mathbb{F}_{q^{2n}}^{\times}\rightarrow \mathbb{F}_{q^n}^{\times}$, has exactly $n$ distinct Galois conjugates, therefore $\chi_{0}$ gives rise to a cuspidal representation $\pi_{0}$ of $\GL_n(\mathbb{F}_{q^2})$ which is the base change of the representation $\pi_{00}$ of $ \GL_n(\mathbb{F}_{q}).$ 

The distinction of $\pi$ by $\Sp_n(D)$ follows from the earlier observation that reduction mod $\pi_D$ of the 
inclusion $\Sp_n(\OO_D)\subset \GL_n(\OO_D)$ is
\[{\rm U}_n(\mathbb{F}_q)\hookrightarrow \GL_n(\mathbb{F}_{q^2}),
\]
together with the well-known fact, Theorem 2 of  \cite{DP99}, that irreducible representations of $ \GL_n(\mathbb{F}_{q^2})$ which are base change from  $\GL_n(\mathbb{F}_{q})$ are distinguished by ${\rm U}_n(\mathbb{F}_q)$.
\end{proof}
\begin{remark}
 \begin{enumerate}
  \item The Langlands parameter of the irreducible representation $\pi= {\rm ind}_{k^{\times}\GL_n(\OO_D)}^{\GL_n(D)}( \pi_{0})$
is of the form $\sigma=\sigma_{0}\otimes {\rm sp}_2$ where $\sigma_{0}$ is the Langlands parameter of the supercuspidal representation of $\GL_n(k)$ compactly induced from the representation $\chi\cdot \pi_{00}$ of $k^{\times}\GL_n(\OO_k)$, 
and ${\rm sp}_2$ is the $2$-dimensional
natural representation of the ${\rm SL}_2(\mathbb C)$ part of the Weil-Deligne group $W_k'=W_k\times {\rm SL}_2(\mathbb C)$ of $k$.

\item If, on the other hand, the cuspidal representation $\pi_{0}$ of $ \GL_n(\mathbb{F}_{q^2})$ is not obtained by base change from 
 $ \GL_n(\mathbb{F}_{q})$ then the Langlands parameter of such a $\pi$ is that of the cuspidal representation of $\GL_{2n}(k)$ which is obtained by compact induction of the representation of $k^{\times}\GL_{2n}(\OO_k)$ which is $\chi$ on $k^{\times}$, and on $\GL_{2n}(\OO_k)$ it corresponds to a representation of $\GL_{2n}(\mathbb{F}_{q})$
which is the automorphic induction of the representation $\pi_{00}$ of $\GL_n(\mathbb{F}_{q^2})$ (and which is cuspidal since we are assuming that the representation $\pi_0$ of $\GL_n(\mathbb{F}_{q^2})$ is not a base change for $\GL_n(\mathbb{F}_{q})$).
\end{enumerate}
\end{remark}
\section{Conjectures on distinction}
The following conjectures have  been proposed by Dipendra Prasad.
\begin{enumerate}
 \item An irreducible discrete series representation $\pi$ of $\GL_n(D_v)$ is distinguished by $\Sp_n(D_v)$ if and only 
if $\pi$ is supercuspidal and the Langlands parameter $\sigma_{\pi}$ of $\pi$ is of the form $\sigma_{\pi}=\tau\otimes {\rm sp}_r$ where $\tau$
 is irreducible and ${\rm sp}_r$ is the $r$-dimensional
natural representation of the ${\rm SL}_2(\mathbb C)$ part of the Weil-Deligne group $W_k'=W_k\times {\rm SL}_2(\mathbb C)$ of $k$ for $r$ even. By Proposition 7.2 below, this is the case if and only if $r=2$, and $n$ is odd. (This is thus exactly the case in which we 
constructed in the last section a supercuspidal representation of $\GL_n(D_v)$ which is distinguished by $\Sp_n(D_v)$.)

\item We follow the notation of Offen-Sayag, Theorem 1 of \cite{OSE07}, to recall that the unitary representations of $\GL_{2k}(F_v)$ which are distinguished by $\Sp_{2n}(F_v)$ are of the form 
\[ \sigma_1 \times \dots \times \sigma_t\times \tau_{t+1}\times \dots \times \tau_{t+s},
\]
where $\sigma_i$ are the Speh representations $U(\delta_i,2m_i)$ for discrete series representations $\delta_i$ of $\GL_{r_i}(F_v)$, and $\tau_i$
are complementary series representations $\pi(U(\delta_i,2m_i),\alpha_i)$ with $\left\vert\alpha_i\right\vert< 1/2$. We suggest that unitary
representations of $\GL_n(D_v)$ distinguished by $\Sp_n(D_v)$ are exactly those representations of $\GL_{n}(D_v)$ which are of the form
\[ \pi=\sigma_1 \times \dots \times \sigma_t\times \tau_{t+1}\times \dots \times \tau_{t+s}\times\mu_{t+s+1}\times \dots \times \mu_{t+s+r},
\]
where
\begin{enumerate}
 \item [(a)] The parameter $\sigma_{\pi}$ of $\pi$ is relevant for $\GL_n(D_v)$, that is, all irreducible subrepresentations of $\sigma_{\pi}$ have even dimension.
\item [(b)] $\sigma_i$ and $\tau_i$ are as in the theorem of Offen-Sayag recalled above.
\item [(c)] $\mu_i$ are supercuspidal representations of $\GL_{m_i}(D_v)$ as in Part (1) of the conjecture.
\end{enumerate}
\item A global automorphic representation of $\GL_n(D_\mathbb{A})$ is distinguished by $\Sp_n(D_\mathbb{A})$ if and only $\JL(\Pi)$ as an automorphic representation of $\GL_{2n}(\mathbb{A}_F)$ (which is same as $\Pi$ at places of $F$ where $D$ splits) is distinguished by $\Sp_{2n}(\mathbb{A}_F)$.
\end{enumerate}

\begin{proposition} The global conjecture in part 3 above implies the local conjecture in part 1.
\end{proposition}
\begin{proof} To prove the Proposition, note that a
discrete series representation $\pi$ of $\GL_n(D_v)$ with parameter
$\tau \otimes {\rm sp}_r$ with $r$ odd is not distinguished by
$\Sp_n(D_v)$ as follows from Theorem 5.5 and the remark 5.6 following it (which assumes validity 
of the globalization theorem of \cite{PS08} for discrete series representations).

Now we prove that a
non-cuspidal discrete series representation $\pi$ of $\GL_n(D_v)$ with parameter
$\tau \otimes {\rm sp}_r$ with $r$ even are not distinguished by
$\Sp_n(D_v)$. Again we will grant ourselves an automorphic representation $\Pi$ of $\GL_n(D_\A)$ which is globally
distinguished by $\Sp_n(D_\A)$. By the Jacquet-Langlands transfer, we get a representation $JL(\Pi)$
of $\GL_{2n}(\A_F)$ which is distinguished by $\Sp_{2n}(\A_F)$, and therefore by the theorem
of Offen-Sayag $JL(\Pi)$ is in the residual spectrum with the Moeglin-Waldspurger type,
$JL(\Pi)= \Sigma \otimes {\rm sp}_d$,  
where $\Sigma$ is a cuspidal  automorphic representation of $\GL_r(\A_F)$ for some integer $r$, and $d$ is a certain even integer; 
here the notation $\Sigma \otimes {\rm sp}_d$ is supposed to denote a certain Speh representation. 
The only option for $d$ in our case is $d =r$, and 
$\Sigma_v = \tau$. By Proposition 7.3 below, we get a contradiction to $\pi$ being a 
non-cuspidal discrete series representation of $\GL_n(D_v)$.

Finally we prove that if we have a cuspidal representation $\pi$ of $\GL_n(D_v)$ with parameter
$\tau \otimes {\rm sp}_r$ with $r$ even, so $r=2$, and $\dim \tau = n$ odd,  then $\pi$ is distinguished by
$\Sp_n(D_v)$.

Construct an automorphic representation of $\GL_n(\A_F)$ whose local component at the place $v$ of $F$
has Langlands parameter $\tau$ with $\dim \tau = n$. Since $\tau$ is an irreducible representation of the Weil group, we are 
considering supercuspidal representation of $\GL_n(F_v)$, and therefore globalization is possible. W​e ​
 moreover assume in this globalization that the global automorphic representation of $\GL_n(\A_F)$ is supercuspidal at all places of $F$ where $D$ is not split. By Moeglin-Waldspurger, this gives an automorphic representation say $\Pi$ of
$\GL_{2n}(\A_F)$ in the residual spectrum, which by the theorems of Offen and Sayag is distinguished by $\Sp_{2n}(\A_F)$. 
By the work of Badulescu,  $\Pi$ can be lifted to $\GL_n(D_{\A})$, which by our global conjecture (3) above is globally 
distinguished by $\Sp_n(D_{\A})$, and therefore locally
distinguished at every place of $F$. 
​It remains to make sure
that in this Jaquet-Langlands ​
​transfer from ​$\GL_{2n}(\A_F)$ to $\GL_n(D_{\A})$, the local representation obtained for $\GL_n(D_v)$ is the cuspidal
representation $\pi$ with parameter $\tau \otimes {\rm sp}_2$; this is forced on us when $\pi$ is cuspidal by lemma 7.4 below. (The representation
$\pi$ could have changed to its Zelevinsky involution, but $\pi$ being cuspidal remains invariant under
the Zelevinsky involution.) \end{proof}

\vspace{2mm}

The following proposition is  due to  Deligne-Kazhdan-Vigneras [2], Theorem B.2.b.1, as well as  Badulescu, proposition 3.7 of [1].

\begin{proposition} A discrete series representation of  $\GL_n(D_v)$,
 where $D_v$ is an arbitrary division algebra over the local field $F_v$, with parameter $\tau \otimes {\rm sp}_r $ is a cuspidal representation
of $\GL_n(D_v)$  if and only if $(r,n)=1$.
\end{proposition}

In the following proposition, we refer to Badulescu [1] for the notion of a $d$-compatible representation of 
$\GL_{nd}(F_v)$.

\begin{proposition} Let $D_v$ be a division algebra over a local field $F_v$ of dimension $d^2$. 
The map $|{\bf LJ}|$ from 
$d$-compatible irreducible admissibile unitary representations of $\GL_{nd}(F_v)$ 
to irreducible unitary representations
of $\GL_n(D_v)$ 
takes a Speh representation associated to a cuspidal representation on $\GL_{nd}(F_v)$ to either a 
cuspidal representation on $\GL_n(D_v)$, or to a Speh representation, i.e.,  the image under $|{\bf LJ}|$
of  a Speh representation associated to a cuspidal representation on $\GL_{nd}(F_v)$ is never a non-cuspidal discrete
series representation on $\GL_n(D_v)$. 
\end{proposition}

\begin{proof} The proof follows from the fact that   $|{\bf LJ}|$ commutes with the Zelevinsky involution, and that 
the Zelevinsky involution of a discrete series representation is itself if and only if the discrete series
 representation is supercuspidal. (We apply this latter fact on $\GL_n(D_v)$.)
\end{proof}

We also had occasion to use the following lemma.

\begin{lemma}
  The map $|{\bf LJ}|$ from 
$d$-compatible irreducible admissibile unitary representations of $\GL_{nd}(F_v)$ to irreducible unitary representations
of $\GL_n(D_v)$ has fibers of cardinality one or two over a discrete series representation of $\GL_n(D_v)$, and if of cardinality two, 
the two elements in the fiber are 
Zelevinsky involution of each other, and the image consists of a cuspidal representation of $\GL_n(D_v)$. 
\end{lemma}

\begin{proof} Assume that we are considering the fibers of the map $|{\bf LJ}|$ over a discrete series representation
of $\GL_n(D_v)$
with Langlands parameter $\tau \otimes {\rm sp}_r $. 
All the representations in the fiber are contained in the 
principal series representation 
$$\tau \nu^{(r-1)/2} \times \tau \nu^{(r-3)/2} \times \cdots \times \tau \nu^{-(r-1)/2}.$$
It is well-known that there are exactly two irreducible unitary representations among sub-quotients of this principal series,
one of which is the Langlands quotient which is a Speh module, and the other the discrete series representation
with parameter $\tau \otimes {\rm sp}_r $, proving the lemma.
 \end{proof}

\bibliographystyle{plain}

\end{document}